\documentclass[oneside,11pt]{amsart}
\pdfoutput=1

\usepackage{amssymb}
\usepackage{amsmath}
\usepackage[pdftex]{graphicx}
\usepackage[hidelinks]{hyperref}
\usepackage{algorithm}
\usepackage{algpseudocode}

\usepackage[font=small]{caption}
\usepackage[a4paper,left = 3cm, right= 3cm,  bottom = 4.5cm]{geometry}
\usepackage{mathtools}
\usepackage{microtype}
\usepackage{xcolor}

\newcommand{\vertiii}[1]{{\left\vert\kern-0.25ex\left\vert\kern-0.25ex\left\vert #1 
    \right\vert\kern-0.25ex\right\vert\kern-0.25ex\right\vert}}

\newtheorem{theorem}{Theorem}
\theoremstyle{definition}
\newtheorem{example}[theorem]{Example}

\numberwithin{theorem}{section}
\numberwithin{equation}{section}

\newcommand{\R}{\mathbb{R}}
\newcommand{\st}{\mathrm{s.t.}}
\newcommand{\sdp}{\mathrm{sdp}}

\newcommand{\svd}{\mathrm{svd}}
\newcommand{\T}{\mathsf{T}}
\newcommand{\op}{T}

\DeclareMathOperator{\tridiag}{tridiag}

\DeclareMathOperator*{\argmin}{argmin}
\DeclareMathOperator{\vectorize}{vec}
\DeclareMathOperator{\matricize}{mat}

\title[]{Kronecker Product Approximation of Operators in Spectral Norm via Alternating SDP}

\author{Mareike Dressler \and Andr\'e Uschmajew \and Venkat Chandrasekaran}

\date{}


\begin{document}

\begin{abstract}
The decomposition or approximation of a linear operator on a matrix space as a sum of Kronecker products plays an important role in matrix equations and low-rank modeling.  The approximation problem in Frobenius norm admits a well-known solution via the singular value decomposition.  However, the approximation problem in spectral norm, which is more natural for linear operators, is much more challenging.  In particular, the Frobenius norm solution can be far from optimal in spectral norm.  We describe an alternating optimization method based on semidefinite programming to obtain high-quality approximations in spectral norm, and we present computational experiments to illustrate the advantages of our approach.
\end{abstract}

\maketitle

\section{Introduction and problem statement}

Let $\op \colon \R^{m \times n} \to \R^{m \times n}$ be a linear operator acting on the space of real $m \times n$ matrices.  Every such operator can be decomposed into a sum of tensor product operators $A_j \otimes B_j$ such that
\begin{equation}\label{eq: formula for TX}
\op(X) = \sum_{j = 1}^r A_j^{} X B_j^\T
\end{equation}
for some matrices $A_j \in \R^{m \times m}$ and $B_j \in \R^{n \times n}$. By choosing a suitable $mn \times mn$ matrix representation of the linear operator $\op$, the above decomposition can be written as
\begin{equation}\label{eq: Kronecker decomposition}
\op = \sum_{j = 1}^r A_j \otimes B_j,
\end{equation}
where now $A_j \otimes B_j$ is the standard Kronecker product of matrices $A_j$ and $B_j$. In the following we will not distinguish between these two interpretations of $\op$.

The smallest $r$ needed for a decomposition~\eqref{eq: Kronecker decomposition} to exist is often called the \emph{Kronecker rank} of $\op$. The Kronecker rank has the (sharp) upper bound
\[
r \le \min(m^2,n^2),
\]
and a minimal decomposition can be found by applying a rank revealing decomposition to a suitable reshape $\widehat \op$ of $\op$ into an $m^2 \times n^2$  matrix~\cite{VanLoan1992,Konstantinov2000}; see section~\ref{sec: svd method}. In~\cite{Konstantinov2000}, the Kronecker rank has also been called the Sylvester index of $\op$.

While the usual rank of $\op \in \R^{mn \times mn}$ measures the number of linear independent columns/rows, the Kronecker rank measures the number of linear independent blocks when partitioning $T$ into $m^2$ blocks of size $n \times n$. The rank and the Kronecker rank are unrelated in general. For instance, if $\op$ consists of the blocks $\op_{\mu \nu} = a_\mu^{} b_\nu^\T$, where $a_1,\dots,a_m$ and $b_1,\dots,b_m$ are linear independent systems in $\R^n$, then $\op$ has ordinary rank one but Kronecker rank $m^2$ (here $m \le n$). Vice versa, when $\op = A \otimes B$ with invertible $A$ and $B$, its rank is $mn$ while the Kronecker rank is only one.

Linear operators with small Kronecker rank play an important role in numerical linear algebra and scientific computing.

In particular, they appear in many linear matrix equations arising in control theory or numerical analysis. While solving a general linear matrix equations such as
\begin{equation}\label{eq: linear matrix equations}
\op(X) = Y
\end{equation}
is numerically challenging if sought matrices $X$ are large, a Kronecker structured equation of the form~\eqref{eq: formula for TX} allows for a more efficient numerical treatment if the number of terms $r$ is not too large. The reason is that the evaluation of the operator $\op$ on a matrix $X$ can be computed by just a few left and right matrix multiplications. This allows the implementation of iterative methods for solving equations such as~\eqref{eq: linear matrix equations} at (relatively) small computational cost, and even more so if additional structure of the matrices $A_j$ and $B_j$ can be exploited.

There is also a second, in a sense more fundamental, reason why operators with small Kronecker rank are of interest in matrix equations: in practice, it can often be observed that the solution $X$ to an equation~\eqref{eq: linear matrix equations}, where $\op$ has small Kronecker rank and the right-hand side $Y$ is a low-rank matrix, can itself be well approximated by low-rank matrices. Thus a low-rank model $X \approx U V^\T$ for the solution is then justified which in turn makes the application of $\op$ via left and right multiplications even more efficient. This is the basis for many efficient low-rank solvers for matrix equations; see~\cite{Simoncini2016} for an overview. The numerical observation can be rigorously proven for certain matrix equations, most notably for Sylvester-type equations
\[
\op(X) = A X + X B^\T = Y
\]
when $A$ and $B$ are symmetric positive definite matrices. It can be shown that the inverse $\op^{-1}$ of this operator can be approximated in operator norm by a sum of $k$ Kronecker products with exponential (or subexponential) error decay in $k$, for instance, by means of exponential sums~\cite{Grasedyck04}. {Applying such an approximate inverse of $\op$ to the right-hand side $Y$ of the matrix equation provides an approximate solution whose rank is at most $k$ times the rank of $Y$, with an error decaying exponentially in $k$.} This even works for tensor versions of Sylvester-type operators. We refer to~\cite{Beylkin2002,Tyrtyshnikov2004,Hackbusch2006a,Hackbusch2006b,Oseledets2009,Hackbusch2019} for applications of Kronecker product approximations of operators and low-rank tensor calculus in scientific computing. 

Motivated by the above considerations, we are interested in the general question of \emph{approximating} a matrix $\op \in \R^{mn \times mn}$ or its inverse by an operator of low Kronecker rank in operator norm. Given $\op$ and a rank bound $k$, we hence consider the optimization problem
\begin{equation}\label{eq: optimization problem}
\min_{A_j, B_j} \Bigg\| \op - \sum_{j = 1}^k A_j \otimes B_j \Bigg\|_2.
\end{equation}
Here $\| \cdot \|_2$ denotes the spectral norm of $mn \times mn$ matrices, which corresponds to the operator norm on $\R^{m \times n}$ with respect to the Frobenius norm. In particular,
\[
\Bigg\| \op - \sum_{j = 1}^k A_j \otimes B_j \Bigg\|_2 = \max_{\| X \|_F = 1} \Bigg\| \op (X) - \sum_{j=1}^k A_j^{} X B_j^\T \Bigg\|_F.
\]
We will also consider the approximation of inverse operators $\op^{-1}$ using a different cost function (section~\ref{sec: generalizations}).

It is known that if in~\eqref{eq: optimization problem} we replace the spectral norm by the Frobenius norm, the approximation problem admits (in principle) a closed-form solution via a singular value decomposition (SVD) of a reshaped $m^2 \times n^2$ matrix $\widehat \op$, which was already mentioned before. This method has been worked out in~\cite{VanLoan1992} and will be explained in section~\ref{sec: svd method}. The main underlying reason that such an approach works for the Frobenius norm is that this norm is invariant under reshaping a matrix. For the spectral norm, this is, however, not the case. So while indeed the SVD also provides a best low-rank approximation in spectral norm of $\widehat \op$ as an operator from $\R^{n^2}$ to $\R^{m^2}$, it will not provide the optimal low Kronecker rank approximation for $\op$ in spectral norm. To our knowledge, the problem~\eqref{eq: optimization problem} has not been addressed in the literature in this general form.

\smallskip
In this work, we consider an alternating optimization method for solving~\eqref{eq: optimization problem} based on semidefinite programming. The underlying idea is based on the well-known fact that the best approximation of a given matrix in spectral norm by an element from an affine linear matrix subspace can be found by solving a semidefinite program (SDP) and is therefore computable in polynomial time to a desired accuracy.  Due to the bilinearity of the Kronecker product, an alternating optimization approach for the unknown matrices $A_j$ and $B_j$ in problem~\eqref{eq: optimization problem} leads to a sequence of approximation problems on linear matrix subspaces, which therefore can be solved via SDPs. Furthermore, by adding regularization one can even guarantee that the solutions to the subproblems are unique and bounded, which plays a role in our convergence analysis.

Such an approach has computational limitations when $m$ or $n$ become large, as it requires solving rather large scale SDPs. However, as a proof of concept, our numerical experiments demonstrate that applying our algorithm can provide significantly better low Kronecker rank approximations to operators $\op$ than the SVD method in certain settings.

\smallskip
The paper is organized as follows. In section~\ref{sec: svd method}, we review the SVD method for computing low Kronecker rank approximations and provide an example showing that this method can be far from optimal with regard to the spectral norm. Section~\ref{sec: alternating sdp approach} presents the alternating SDP approach for solving the spectral norm approximation problem~\eqref{eq: optimization problem}. We also discuss a regularized version of the problem in section~\ref{subsec: regularization}, which leads to the general form of Algorithm~\ref{alg:Alternating}. Moreover, we briefly touch upon how the problem of approximating inverses could be handled (section~\ref{sec: generalizations}). In section~\ref{sec: convergence}, we provide some convergence statements based on results for biconvex optimization from the literature. Finally, in section~\ref{sec: experiments}, we present computational experiments illustrating the benefit of our proposed approach. We conclude with a brief discussion and an outlook on potential future work; see section~\ref{sec: conclusion-outlook}.

\section{The SVD method}\label{sec: svd method}

The SVD solution to the Kronecker product approximation problem has been proposed in~\cite{VanLoan1992}; see also~\cite{VanLoan2000,Konstantinov2000}. It works by rearranging the matrix $\op \in \R^{mn \times mn}$ into another matrix $\widehat \op \in \R^{m^2 \times n^2}$. Specifically, from a block partition 
\[
\op = \begin{bmatrix} T_{11} & T_{12} & \cdots \\
T_{21} & \ddots & \\
\vdots & &
\end{bmatrix}
\]
with $m^2$ blocks $T_{\mu \nu} \in \R^{n \times n}$ one constructs the matrix $\widehat \op$ with rows $\vectorize(T_{\mu \nu})^\T$ in the ordering
\[
\widehat \op = \begin{bmatrix} -\vectorize(T_{11})^\T- \\ -\vectorize(T_{21})^\T- \\ \vdots \\ -\vectorize(T_{12})^\T-  \\ \vdots \end{bmatrix}.
\]
It is then easy to see that the Kronecker product decomposition~\eqref{eq: Kronecker decomposition} of $\op$ is equivalent to a decomposition
\[
\widehat \op = \sum_{j=1}^r \vectorize(A_j) (\vectorize(B_j))^\T
\]
into rank-one matrices. This shows that a (usual) low-rank decomposition or approximation of $\widehat \op$ yields decompositions or approximations of $\op$ by sums of Kronecker products. Clearly the transformation $\op \leftrightarrow \widehat \op$ is an isometry between $\R^{mn \times mn}$ and $\R^{m^2 \times n^2}$ in Frobenius norm, since it only rearranges entries of $\op$. 

The required low-rank approximation of $\widehat \op$ can be achieved, for instance, using SVD. This provides the method displayed as Algorithm~\ref{alg:svd}, where {$\matricize_{m}(\cdot)$} denotes the inverse of the $\vectorize(\cdot)$ operation {for $m \times m$ matrices $A$, that is, $\matricize_m(\vectorize(A)) = A$}. In the algorithm, we included an optional step for optimizing the scaling of the truncated SVD with respect to spectral norm approximation. It could be solved using semidefinite programming similar to the methods derived in section~\ref{sec: alternating sdp approach}.

\begin{algorithm}[h]
\caption{SVD method~\cite{VanLoan1992}}\label{alg:svd}
\begin{algorithmic}
\Require Matrix $\op \in \R^{mn \times mn}$, dimensions $m,n$, approximation rank $k$
\Ensure Kronecker rank-$k$ approximation $\op_{\svd}$ of $\op$
\State Construct the matrix $\widehat \op$ from $\op$.
\State Compute SVD: $\widehat \op = U \Sigma V^\T = \sum_{j=1}^r \sigma_j^{} u_j^{} v_j^\T$.
\State Set $A_j = {\matricize_m}(\sigma_j^{1/2} u_j), \quad B_j = {\matricize_n}(\sigma_j^{1/2} v_j)$.
\State Set $\alpha_1 = \dots = \alpha_k = 1$.
\State \textbf{Optional:} Find $(\alpha_1,\dots,\alpha_k) = \displaystyle \argmin_{\alpha_j \in \R} \Bigg\| \op -  \sum_{j=1}^k \alpha_j A_j \otimes B_j \Bigg\|_2$. \Comment{Optimal scaling}
\State \Return $\op_{\svd} \coloneqq \sum_{j=1}^k \alpha_j A_j \otimes B_j$
\end{algorithmic}
\end{algorithm}

It is now interesting to observe that the approximation $\sum_{j=1}^k A_j \otimes B_j$ of $\op$ obtained by the above procedure is optimal in Frobenius norm; that is, it solves a version of problem~\eqref{eq: optimization problem} in Frobenius norm instead of the spectral norm. The reason is that {by the Eckart--Young--Mirsky theorem} $\sum_{j=1}^k \sigma_j^{} u_j^{} v_j^\T$ is an optimal rank-$k$ approximation of $\widehat \op$ in Frobenius norm and the transformation $\op \leftrightarrow \widehat \op$ is an isometry.  
More generally, {the Eckart--Young--Mirsky theorem actually states that} the truncated SVD of $\widehat \op$ is optimal in any unitarily invariant norm on $\R^{m^2 \times n^2}$; in particular, it is the best approximation in spectral norm on $\R^{m^2 \times n^2}$. Therefore, $\sum_{j=1}^k A_j \otimes B_j$ is also an optimal approximation of $\op$ in the norm defined via
\[
\vertiii{\op} \coloneqq \| \widehat \op \|_2
\]
on $\R^{mn \times mn}$. Unfortunately, in general, this norm (and other possible ones obtained in this way) does not equal the spectral norm in $\R^{mn \times mn}$. For example, for $\op = \linebreak I \otimes I \leftrightarrow \widehat \op = \vectorize(I_m) \vectorize(I_n)^\T$ (identity matrices) we have $\| \op \|_2 = 1$, but $\vertiii{ \op } = \| \widehat \op \|_2 = \| \vectorize(I_m) \|_2 \| \vectorize(I_n) \|_2 = \sqrt{mn}$.

Correspondingly, the SVD approach will usually not lead to optimal solutions of the initial approximation problem for $\op$ in spectral norm, even if the scaling is optimized. It is in fact not so difficult to construct such counterexamples. 

\begin{example}\label{ex: counterexample}
Let 
\begin{equation}\label{eq: T counterexample}
\op=\sigma_1 A_1 \otimes B_1 +  A_2\otimes B_2
\end{equation}
such that 
$\langle A_1,A_2\rangle = 0$, $\langle B_1,B_2\rangle = 0$, $\|A_1\|_F = \|A_2\|_F = \|B_1\|_F = \|B_2\|_F = 1$, and $\sigma_1 > 1$. Here $\langle \cdot, \cdot \rangle$ denotes the Frobenius inner product. Assume the goal is to find a Kronecker rank-one approximation, i.e., $k=1$. Note that by construction~\eqref{eq: T counterexample} already corresponds to an SVD of $\widehat \op$. Hence Algorithm~\ref{alg:svd} proposes $\op_{\svd} = \alpha_1 \sigma_1 A_1 \otimes B_1$ as an approximation to $\op$. The error is
\[
\op - \op_{\svd} = \sigma_1(1 - \alpha_1) A_1 \otimes B_1 + A_2 \otimes B_2.
\]
We now choose $A_2 = a a^\T$ and $B_2 = b b^\T$ to be rank-one matrices (the symmetry is not essential) {with $\| a \|_2 = \| b \|_2=1$}, and $A_1, B_1$ to be matrices with $A_1 a = 0$ and $B_1 b = 0$. Applying $T$ to the rank-one matrix $a b^\T$ we then have that
\[
(\op - \op_{\svd})(a b^\T) = \sigma_1(1 - \alpha_1) (A_1 a) (B_1 b)^\T + a a^\T a b^\T b b^\T = a b^\T
\]
for any choice of $\alpha_1$, {using that $A_1 a = 0$ and $a^\T a = b^\T b = 1$}. This shows that
\[
\| \op - \op_{\svd} \|_2 \ge 1.
\]
(Actually, when $\alpha_1$ is picked optimally as assumed in Algorithm~\ref{alg:svd}, one also has $\| \op - \op_{\svd} \|_2 \le \| A_2 \otimes A_1 \|_2 = 1$, hence equality.) On the other hand, using $ A_2 \otimes B_2$ as an approximation would lead to
\[
\| \op - A_2 \otimes B_2 \|_2 = \| \sigma_1 A_1 \otimes B_1 \|_2 = \sigma_1 \| A_1 \|_2 \| B_1 \|_2
\]
using properties of the Kronecker product. If we take $A_1$ and $B_1$ to be of rank $m-1$ and $n-1$, respectively, with identical (nonzero) singular values $\frac{1}{\sqrt{m-1}}$ and $\frac{1}{\sqrt{n-1}}$, we obtain
\begin{equation}\label{eq: errorbound}
\| \op - A_2 \otimes B_2 \|_2 = \sigma_1 \frac{1}{\sqrt{(m-1)(n-1)}}.
\end{equation}
For $\sigma_1$ close to one and $m,n$ large this is much smaller than one, demonstrating that the SVD solution can be rather bad. Note that in this construction (assuming $\sigma_1 \frac{1}{\sqrt{(m-1)(n-1)}} < 1$), one can show that $\| \op \|_2 = 1$, so the above estimates measure the relative error. The SVD solution is hence not better than taking the zero approximation.
\end{example}

The above example is designed artificially to make the SVD method fail to provide a reasonably low Kronecker rank approximation in spectral norm, which clearly motivates the need for alternative methods. Of course, in practical instances the SVD solution might not be as bad. This is also observed in the numerical experiments in section~\ref{sec: experiments}, where the SVD method provides useful results that can still can be significantly improved.

\section{Alternating SDP approach}\label{sec: alternating sdp approach}

In this main section we describe an approach for solving the problem~\eqref{eq: optimization problem} based on alternating optimization and semidefinite programming.

\subsection{Basic idea}\label{subsec: basic idea sdp}

We use the fact that a minimization of the operator norm on an affine linear space of matrices can be turned into an SDP; see~\cite[section~4.6.3]{Boyd:Vandenberghe:ConvOpt}. This is based on the observation that for a matrix an inequality $\| S \|_2 \le \tau$ for the spectral norm is equivalent with $S^\T S \preceq \tau^2 I$ and $I$ being the identity matrix. Hence the problem~\eqref{eq: optimization problem} can be first reformulated as
\begin{equation*}
\begin{gathered}
 \min_{\tau \ge 0,\, A_j,\, B_j}\quad  \tau \\
 \st\quad   \left(\op-\sum_{j = 1}^k A_j \otimes B_j\right)^{\!\!\T} \left(\op-\sum_{j = 1}^k A_j \otimes B_j \right) \preceq \tau^2 I.   
\end{gathered}
\end{equation*}
By applying the Schur complement this turns into
\begin{equation}\label{eq: pre-SDP}
\begin{gathered}
 \min_{\tau,A_j,B_j}\quad  \tau \\
 \st\quad  M(\tau,A,B) \succeq 0
\end{gathered}
\end{equation}
with
\begin{equation*}\label{eq: Schur constraint}
    M(\tau,A,B) \coloneqq \begin{bmatrix}
\tau I & \op-\sum_{j = 1}^k A_j \otimes B_j \\
(\op-\sum_{j = 1}^k A_j \otimes B_j)^\T & \tau I 
\end{bmatrix}.
\end{equation*}
Note that the condition $\tau \ge 0$ has been dropped in this second version since it is automatically implied by the positive semidefiniteness of $M(\tau,A,B)$.

Since $M(\tau,A,B)$ is affine linear in each of the (block) variables $\tau$, $A = (A_1,\dots,A_k)$, and $B = (B_1,\dots,B_k)$, we have achieved that the optimization problem~\eqref{eq: pre-SDP} for each of these block variables, when the others are kept fixed, is an SDP. In fact, the problem for the auxiliary variable $\tau$ can be subsumed into the other ones: we have an SDP in $(\tau,A)$ when $B$ is fixed, or in $(\tau,B)$ when $A$ is fixed. As SDPs can be solved in polynomial time to a desired accuracy, this suggests an alternating optimization strategy to tackle the initial problem~\eqref{eq: optimization problem} through a sequence of SDPs: we first fix $B=(B_1,\ldots,B_k)$ and optimize $A=(A_1,\ldots,A_k)$ (and $\tau$), and then we fix $A$ to find the optimal $B$ (and $\tau$). We describe a generalization of this procedure in what follows with additional regularization parameters in Algorithm~\ref{alg:Alternating}. By construction the subproblems for $A$ and $B$ always admit at least one solution, since from~\eqref{eq: optimization problem} it follows that they are equivalent with finding a best approximation of $\op$ in operator norm on the \emph{linear} subspaces $\{ \sum_{j=1}^k A_j \otimes B_j \colon A_1,\dots,A_k \in \R^{m \times m} \}$ and $\{ \sum_{j=1}^k A_j \otimes B_j \colon B_1,\dots,B_k \in \R^{n \times n} \}$ of $\R^{mn \times mn}$, respectively.

\subsection{Adding regularization}\label{subsec: regularization}

While the subproblems for $A$ and $B$ in the alternating optimization approach outlined above are guaranteed to admit optimal solutions, neither uniqueness nor {a uniform} boundedness of these solutions seems to be easy to guarantee in general. As we will see in section~\ref{sec: convergence}, this poses a problem in the convergence analysis of alternating block minimization problems. Both issues can be addressed by adding regularization to the approach outlined in section~\ref{subsec: basic idea sdp}, albeit at the expense of larger SDPs.

For convenience, we define 
\[
F(A,B) \coloneqq \Bigg\| \op - \sum_{j = 1}^k A_j \otimes B_j \Bigg\|_2,
\]
where again $A = (A_1,\dots,A_k)$ and $B = (B_1,\dots,B_k)$. Instead of~\eqref{eq: optimization problem} we  then consider
\begin{equation}\label{eq: regularized problem}
\min F_{\lambda,\mu}(A,B) = 
F(A,B) + \lambda \sum_{j=1}^k \| A_j \|^2_F + \mu \sum_{j=1}^k \| B_j \|_F^2
\end{equation}
with regularization parameters $\lambda, \mu \ge 0$. If $\lambda > 0$, then for fixed $B$ the function $A \mapsto F_{\lambda,\mu}(A,B)$ is strictly convex and coercive. A similar statement holds for the function $B \mapsto F_{\lambda,\mu} (A,B)$ in the case $\mu > 0$. Therefore, if $\lambda$ and $\mu$ are positive, the subproblems now admit unique solutions. Moreover, the sequences of solutions can easily be bounded (since the function values of $F_{\lambda,\mu}$ are decreasing and $F$ is bounded from below).

We claim that in the regularized problem the subproblems for $A$ and $B$ (and likewise for single $A_j$ and $B_j$) can again be turned into SDPs as follows. Assume $B$ is fixed and $A$ should be optimized. Introducing slack variables $\gamma_0,\gamma_1,\dots,\gamma_k$ we first rewrite~\eqref{eq: regularized problem} as
\begin{equation*}
\begin{gathered}
 \min_{\tau,\,\gamma_0,\,\gamma_1,\ldots,\gamma_k, A} \quad  \tau \\
 \begin{aligned}
 \st  \quad  F(A, B) &\leq \gamma_0,\\
   \lambda \| A_j \|_F^2 &\leq \gamma_j, \; j=1,\ldots, k,\\
   \gamma_0 + \gamma_1 + \cdots + \gamma_k &\leq \tau.
 \end{aligned}
\end{gathered}
\end{equation*}

As in~\eqref{eq: pre-SDP}, the first constraint $F(A,B) \le \gamma_0$ can be turned into $M(\gamma_0,A,B) \succeq 0$. The constraints $\lambda \| A_j \|_F^2 \le \gamma_j$ are equivalent to
\begin{equation}\label{eq:Schur constr}
   N_\lambda(\gamma_j,A_j) \coloneqq  \begin{bmatrix}
\gamma_j & \sqrt{\lambda} \vectorize(A_j) \\
\sqrt{\lambda} \vectorize(A_j)^\T &  I 
\end{bmatrix} \succeq 0
\end{equation}
 due to the Schur complement. Finally note that the last constraint is $\tau-\gamma_0- \cdots-\gamma_k\ge 0$. We hence arrive at the following SDP:

\begin{equation}\label{eq: SDP for A with regularization}
\begin{gathered}
 \min_{\tau,\,\gamma_0,\,\gamma_1,\ldots,\gamma_k, A} \quad  \tau \\
 \begin{aligned}
 \st \quad M(\gamma_0,A,B) &\succeq 0,\\
   N_\lambda(\gamma_j,A_j) &\succeq 0, \ j=1,\ldots, k,\\
   \tau-\gamma_0- \cdots-\gamma_k&\ge 0.
 \end{aligned}
\end{gathered}
\end{equation}
Note that the constraints automatically imply $\gamma_0,\gamma_1,\dots,\gamma_k \ge 0$, and thus also $\tau \ge 0$.

When $A$ is fixed and we aim to optimize for $B$, we have to solve the analogous problem
\begin{equation}\label{eq: SDP for B with regularization}
\begin{gathered}
 \min_{\tau,\,\delta_0,\,\delta_1,\ldots,\delta_k, B} \quad  \tau \\
 \begin{aligned}
 \st \quad M(\delta_0,A,B) &\succeq 0,\\
   N_\mu(\delta_j,B_j) &\succeq 0, \ j=1,\ldots, k,\\
   \tau-\delta_0- \cdots-\delta_k&\ge 0.
 \end{aligned}
\end{gathered}
\end{equation}
Here we have slightly abused notation since $N_\mu(\delta_j,B_j)$ might be of different size than defined in~\eqref{eq:Schur constr}. Of course, when $\lambda$ or $\mu$ are zero, the corresponding constraints in~\eqref{eq: SDP for A with regularization} and~\eqref{eq: SDP for B with regularization} can be simply omitted and the problem reduces to the basic idea explained in section~\ref{subsec: basic idea sdp}. 

The resulting alternating SDP algorithm is summarized as Algorithm~\ref{alg:Alternating}.

\begin{algorithm}[h]
\caption{Alternating SDP method (ASDP)}\label{alg:Alternating}
\begin{algorithmic}
\Require Operator $\op$, dimensions $m,n$,  initial $B_0=(B_{01},\ldots,B_{0k})$, approximation rank~$k$, regularization parameters $\lambda,\mu \ge 0$, number of iterations $N_{\text{outer}}$
\Ensure Kronecker rank-$k$ approximation $\op_{\sdp}$ of $\op$ in spectral norm
\State $B \gets B_0$
\For{$N_{\text{outer}}$ iterations}
    \State Update $A$ by solving the SDP~\eqref{eq: SDP for A with regularization} with current $B$ fixed.
    \State Update $B$ by solving the SDP~\eqref{eq: SDP for B with regularization} with current $A$ fixed.
\EndFor
\State \Return $\op_{\sdp}:=\sum_{j = 1}^k A_j \otimes B_j$
\end{algorithmic}
\end{algorithm}

\subsection{Approximation of inverse operators}\label{sec: generalizations}

As outlined in the introduction, in several applications one is actually interested in the approximation of the inverse $\op^{-1}$ of a given operator $\op$ by sums of Kronecker products. {We note that there is no obvious relation between low Kronecker rank approximations of $\op$ and $\op^{-1}$. Of course, when $\op = A \otimes B$ has Kronecker rank one and is invertible (with $A$ and $B$ square), then necessarily $A$ and $B$ are invertible and $T^{-1} = A^{-1} \otimes B^{-1}$. Apart from that, it is often numerically observed that when $\op$ has small Kronecker rank, then $\op^{-1}$ is \emph{well approximable} by low Kronecker rank. For Sylvester-type operators $\op = A \otimes I + I \otimes B$ (Kronecker rank two) this can be rigorously proven~\cite{Grasedyck04}; this is discussed for the special case of Lyapunov operator in section~\ref{sec: inverse Lyapunov}.}

Since the goal would be to find an approximation of $\op^{-1}$ without forming it explicitly, we cannot apply the above methods directly. A natural way is to consider the modified optimization problem 
\begin{equation}\label{eq: inverse approximation}
\min_{A_j, B_j} \left\| I - \op \cdot\left(\sum_{j = 1}^k A_j \otimes B_j\right) \right\|_2
\end{equation}
instead and tackle it via alternating optimization. Here $\cdot$ is the matrix product. Note that doing this for the Frobenius norm would lead to an alternating least squares algorithm. For the spectral norm, we rely once again on semidefinite programming.

The derivation of an ASDP method for~\eqref{eq: inverse approximation} is almost analogous to the one developed above. Instead of~\eqref{eq: pre-SDP}, we now have to solve
\begin{equation}\label{eq: generalization-SDP}
\begin{aligned}
 \min_{\tau,A_j,B_j}\quad & \tau \\
 \st\quad &  M_{\text{inv}}(\tau,A,B) \succeq 0
\end{aligned}
\end{equation}
with 
\begin{equation}\label{eq: M_inv}
     M_{\text{inv}}(\tau,A,B) \coloneqq \begin{bmatrix}
\tau I & I-\op \cdot(\sum_{j = 1}^k A_j \otimes B_j) \\
(I-\op\cdot(\sum_{j = 1}^k A_j \otimes B_j))^\T & \tau I 
\end{bmatrix}.
\end{equation}
Since $M_{\text{inv}}$ is affine linear in $\tau$, $A$, and $B$, this problem can again be tackled via alternating SDPs.

It also possible to include regularization in the same way as in~\eqref{eq: regularized problem}. The resulting SDP problems in the alternating optimization approach read the same as~\eqref{eq: SDP for A with regularization} and~\eqref{eq: SDP for B with regularization}, except with $M(\tau,A,B)$ being replaced with $M_{\text{inv}}(\tau,A,B)$.

Note that in such an implicit formulation to approximate the inverse one does in principle not need $\op$ in the full format but only the ability to apply $\op$ to operators of the form $\sum_{j = 1}^k A_j \otimes B_j$ several times. While without any structure this poses a potential computational bottleneck, it can be handled in the situation when $\op$ itself is a sum of Kronecker products, 
\[
\op=\sum_{J=1}^r C_J\otimes D_J.
\]
Then $M_{\text{inv}}(\tau,A,B)$ becomes 
\begin{equation}\label{eq: M_inv efficient}
     M_{\text{inv}} = \begin{bmatrix}
\tau I & \displaystyle{I-\sum_{J=1}^r \sum_{j = 1}^k (C_J A_j) \otimes (D_J B_j)} \\
\displaystyle{\left(I-\sum_{J=1}^r \sum_{j = 1}^k (C_J A_j) \otimes (D_J B_j)\right)^{\!\!\T}} & \tau I 
\end{bmatrix}.
\end{equation}
If $k$ and $r$ are small, the double sums can be efficiently computed. 

We point out that the implicit formulation~\eqref{eq: inverse approximation} for approximating inverses comes at a price in the case of badly conditioned operators. On the one hand,~\eqref{eq: inverse approximation} gives us full control on the relative approximation error of $\op^{-1}$ as follows: assume we have achieved
\[
\left\| I - \op \cdot\left(\sum_{j = 1}^k A_j \otimes B_j\right) \right\|_2 \le \varepsilon;
\]
then
\[
\frac{\left\| \op^{-1} - \sum_{j=1}^k A_j \otimes B_j \right\|_2}{\| \op^{-1}\|_2} \le \frac{\| \op^{-1} \|_2 \left\| I - \op \cdot\left(\sum_{j = 1}^k A_j \otimes B_j\right) \right\|_2}{\| \op^{-1} \|_2} \le \varepsilon.
\]
On the other hand, when the goal is to compute approximate solutions of linear matrix equations such as~\eqref{eq: linear matrix equations} by replacing $\op^{-1}$ with its approximation, then the absolute error or the relative error with respect to $\| \op \|_2$ would be of more relevance. We have
\[
\frac{\left\| \op^{-1} - \sum_{j=1}^k A_j \otimes B_j \right\|_2}{\| \op \|_2} \le \frac{\| \op^{-1} \|_2}{\| \op \|_2}\; \varepsilon,
\]
so this error estimate deteriorates with a bad condition number. Since the case of ill conditioned operators appears frequently in numerical analysis, this may put some additional limitations on this approach. In section~\ref{sec: inverse Lyapunov} we conduct some numerical experiments for approximating inverses based on formulation~\eqref{eq: inverse approximation}.

\section{Convergence}\label{sec: convergence}

Algorithm~\ref{alg:Alternating} realizes a block coordinate optimization method for the {nonconvex} cost function
\[
F_{\lambda,\mu}(A,B) = \Bigg\| \op - \sum_{j = 1}^k A_j \otimes B_j \Bigg\|_2 + \lambda \sum_{j=1}^k \| A_j \|^2_F + \mu \sum_{j=1}^k \| B_j \|_F^2
\]
in problem~\eqref{eq: regularized problem} by sequentially setting block variables to restricted global minima:
\begin{equation}\label{eq: BCD formulation}
\begin{aligned}
A^{(\ell+1)} \quad  &\leftarrow  \quad \argmin_A F_{\lambda,\mu}(A,B^{(\ell)}),\\
B^{(\ell+1)} \quad &\leftarrow \quad \argmin_B F_{\lambda,\mu}(A^{(\ell+1)},B).
\end{aligned}
\end{equation}
The convergence of such block coordinate methods is in general not easy to deduce and requires specific assumptions. For differentiable cost functions it can be shown that if the subproblems are guaranteed to admit unique minima, then all accumulation points of the iterates must be critical points; see, e.g.,~\cite[Proposition~3.7.1]{Bertsekas2016}. However, $F_{\lambda,\mu}$ above is not a differentiable function, since the spectral norm is not. {However, it is still a continuous biconvex function which means that the restriction to one of the block variables $A$ or $B$ is always a convex function. The convergence of of~\eqref{eq: BCD formulation} for such functions has been studied in~\cite{Tseng2001} and~\cite{Gorski2007} on which we rely in what follows.}

Fixed points $(A^*,B^*)$ of the procedure~\eqref{eq: BCD formulation} are characterized by the properties
\[
F_{\lambda,\mu}(A^*,B^*) \le F_{\lambda,\mu}(A,B^*) \quad \text{and} \quad F_{\lambda,\mu}(A^*,B^*) \le F_{\lambda,\mu}(A^*,B)
\]
for all $A$ and $B$, respectively. Such points are called \emph{partial optima} (of the function $F_{\lambda,\mu}$). By~\cite[Theorem~5.1]{Tseng2001}, any accumulation point $(A^*,B^*)$ of the sequence $(A^{(\ell)},B^{(\ell)})$ generated by~\eqref{eq: BCD formulation} for the biconvex function $F_{\lambda,\mu}$ will be indeed a partial optimum (i.e., a fixed point of~\eqref{eq: BCD formulation}) if the restricted functions $A \mapsto F_{\lambda,\mu}(A,B)$ and $B \mapsto F_{\lambda,\mu}(A,B)$ are hemivariate (which means they are not constant on any line segment) and if, in addition, $F_{\lambda,\mu}$ has bounded sublevel sets. Both properties are ensured when $\lambda > 0$ and $\mu > 0$. The restricted functions are then even strictly convex and coercive, implying that the subproblems have unique minima.

In~\cite[Theorem~4.9]{Gorski2007} a similar but slightly weaker condition is required to have the same conclusion: for continuous biconvex functions any accumulation point $(A^*,B^*)$ of a sequence $(A^{(\ell)},B^{(\ell)})$ generated by~\eqref{eq: BCD formulation} is indeed a partial optimum under the assumptions that (i) the iterates are bounded and (ii) at every accumulation point $(A^*,B^*)$ it holds that either $A \mapsto F_{\lambda,\mu}(A,B^*)$ or $B \mapsto F_{\lambda,\mu}(A^*,B)$ has a unique minimizer. However, in order to ensure both conditions we would again require $\lambda > 0$ and $\mu > 0$. In this case, both restricted functions have unique minimizers for which~\cite[Theorem~4.9]{Gorski2007} additionally states that $(A^{(\ell+1)}, B^{(\ell+1)}) - (A^{(\ell)}, B^{(\ell)}) \to 0$.

From the discussion above we conclude that if regularization is present we have the following convergence result for Algorithm~\ref{alg:Alternating}.

\begin{theorem}
Assume $\lambda > 0$ and $\mu > 0$. Then the sequence $(A^{(\ell)},B^{(\ell)})$ generated by Algorithm~\ref{alg:Alternating} possesses at least one accumulation point $(A^*,B^*)$. Further, every accumulation point is a partial optimum of $F_{\lambda,
\mu}$ and achieves the same function value $F_{\lambda,\mu}(A^*,B^*)$. It holds that $(A^{(\ell+1)}, B^{(\ell+1)}) - (A^{(\ell)}, B^{(\ell)}) \to 0$. 
\end{theorem}

\begin{proof}
Since the sublevel sets of $F_{\lambda,\mu}$ are bounded when $\lambda > 0$ and $\mu >0$, and $F_{\lambda,\mu}(A^{(\ell)},B^{(\ell)})$ is monotonically decreasing, there exists at least one accumulation point, and all accumulation points take the same function value. For the statements on the accumulation points we refer to~\cite[Theorem~5.1]{Tseng2001} and~\cite[Theorem~4.9]{Gorski2007} as discussed above.
\end{proof}

For clarity, we point out that in the case $\lambda = 0$ or $\mu = 0$ the solvability of the substeps is still ensured, since, e.g., the update of $A$ realizes a best approximation of $\op$ in operator norm on the linear subspace of all $\sum_{j=1}^k A_j \otimes B_j$ with the $B$ being fixed. However, we cannot guarantee uniqueness in the subproblems in this case. Moreover, we are unable to even ensure boundedness of the iterates, although we do not have a counterexample at hand.

Let us also briefly discuss whether partial optima of $F_{\lambda,\mu}$ are critical points. Clearly, if $F_{\lambda,\mu}$ is differentiable at a partial optimum $(A^*,B^*)$, then $\nabla F_{\lambda,\mu}(A^*,B^*) = 0$. This condition boils down to $\| \op - \sum_{j=1}^k A_j \otimes B_j \|_2$ being differentiable in $(A^*,B^*)$. It is well known that the matrix spectral norm function
\[
S \mapsto \|S\|_2 = \max_{\| x \|_2 = 1} \| Sx \|_2
\]
is differentiable at such $S$ for which the maximum on the right-hand side is achieved for a unique $x$. This follows from a general result on max-functions~\cite{Clarke1975}. Alternatively, let $s_1 \ge s_2 \ge \dots \ge 0$ denote the singular values of $S$; then the spectral norm $\| S \|_2 = s_1$ is (continuously) differentiable at $S$ if $s_1 > s_2$~\cite{Stewart1990}. This leads to the following statement: if at a partial optimum $(A^*,B^*)$ of $F_{\lambda,\mu}$ the matrix $S = \op - \sum_{j=1}^k A_j \otimes B_j$ has a unique largest singular value, then $(A^*,B^*)$ is a critical points of $F_{\lambda,\mu}$ in the sense that $\nabla F_{\lambda,\mu}(A^*,B^*) = 0$. Unfortunately, at the moment we do not have any alternative characterization of this property in terms of structural properties of $\op$.

\section{Numerical experiments}\label{sec: experiments}

In this section, we present results of computational experiments showcasing the potential advantages of the proposed alternating SDP approach. The algorithms have been coded in {MATLAB}. The SDPs are implemented using the {YALMIP} toolbox~\cite{Yalmip} with the SDP solver {SeDuMi}~\cite{SeDuMi}. An advantage of this toolbox is that the SDPs can be parsed almost directly in the given forms such as~\eqref{eq: SDP for A with regularization} and~\eqref{eq: SDP for B with regularization}. However, solving those SDPs becomes computationally expensive in larger dimensions, and no attempts at a more efficient implementation have been made. This explains the small values for $m$ and $n$ in the following experiments. One should also keep in mind that the nonconvex and nonsmooth optimization task~\eqref{eq: optimization problem} is far from trivial even in small dimensions. Our goal here is a proof of concept that the alternating SDP approach allows for better approximations in spectral norm than methods based on Frobenius norm.

\subsection{Illustration of Example~\ref{ex: counterexample}}\label{sec: experiment1}

As a first experiment we simulate the construction given in Example~\ref{ex: counterexample} in which the SVD solution is particularly poor. We take $m=n$ and construct
\begin{equation}\label{eq: example operator}
\op = \sigma_1 A_1 \otimes B_1 +  A_2\otimes B_2
\end{equation}
with $\sigma_1 = 1.9$ and
\begin{equation*}
  A_1 = B_1 = \frac{1}{\sqrt{m-1}}\begin{pmatrix}
  I_{m-1} & 0\\
  0 & 0
 \end{pmatrix}, \quad 
 A_2 = B_2 =  \begin{pmatrix}
  0_{m-1} & 0 \\
  0 & 1
 \end{pmatrix}.
\end{equation*}
Here $I_{m-1}$ and $0_{m-1}$ are the $(m-1)\times (m-1)$ identity and zero matrix, respectively. The goal is to solve~\eqref{eq: optimization problem} with $k=1$, that is, to minimize $\| \op - A \otimes B \|_2$.

As explained in Example~\ref{ex: counterexample}, the SVD method in Algorithm~\ref{alg:svd} will select a solution $\op_{\svd} = \alpha_1 \sigma_1 A_1 \otimes A_2$, but the error will be constant, $\| \op - \op_{\svd} \|_2 = 1$. Note that for $m = n \ge 3$ this is also the relative error, since $\| \op \|_2 =1$ then. On the other hand, an optimal approximation should yield an error less than $\| \op - A_2 \otimes B_2 \|_2 = \sigma_1 / (m-1)$.

\begin{figure}[t]
\includegraphics[width=.67\textwidth]{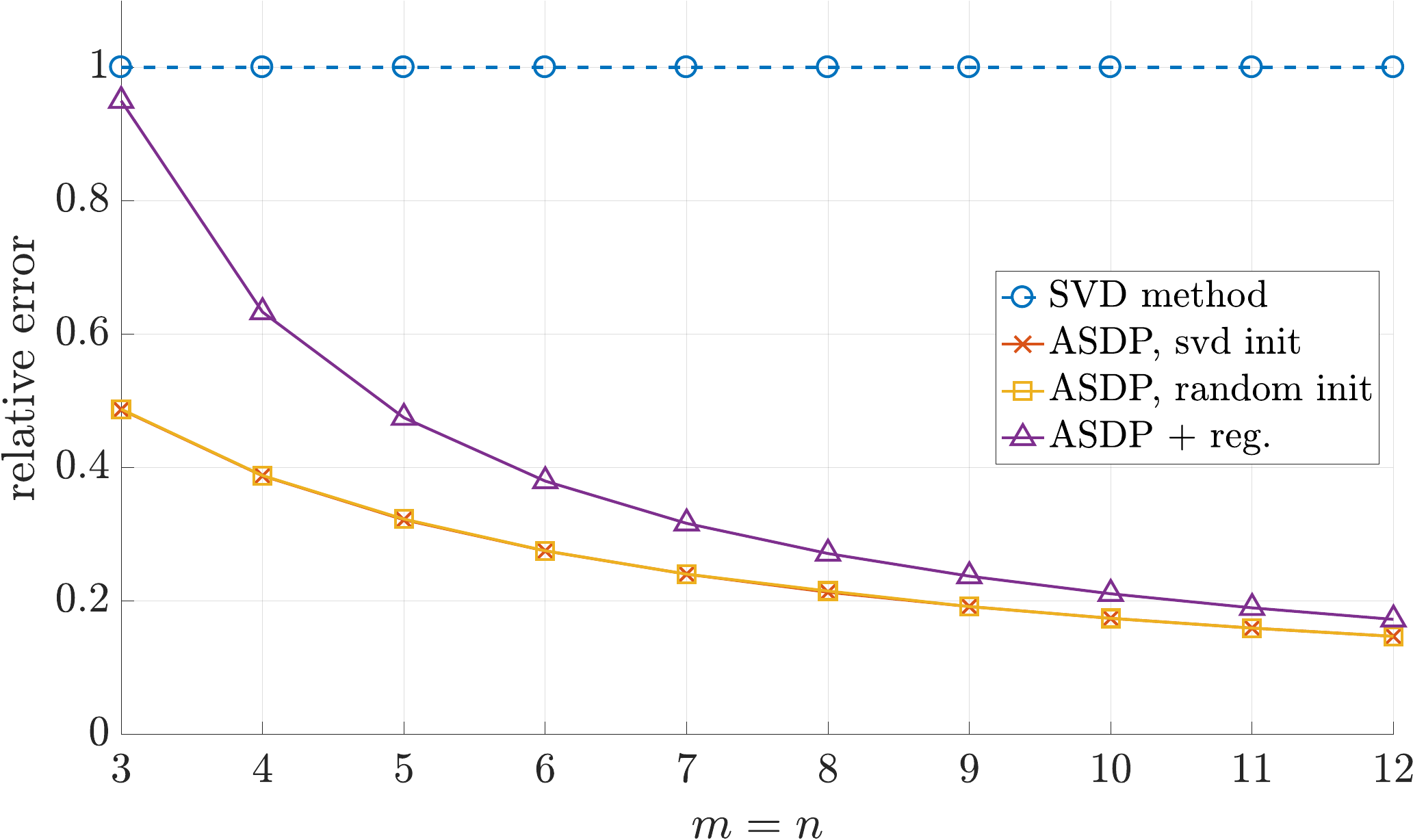}
\caption{Results of Algorithm~\ref{alg:svd} (SVD method) and Algorithm~\ref{alg:Alternating} (ASDP) with $k=1$ for the operator $\op$ in~\eqref{eq: example operator} with $m = n = 3,\dots,10$. The $y$-axis shows the error $\| \op - A \otimes B \|_2$ for the computed approximation $A \otimes B$. Note that $\| \op \|_2 = 1$.
{The red and yellow curves agree.} }
\label{fig: experiment1}
\end{figure}

In Figure~\ref{fig: experiment1} we compare the achieved error $\| \op - A \otimes B \|_2$ of the SVD method computed by Algorithm~\ref{alg:svd} (dashed line) with ASDP solutions obtained with Algorithm~\ref{alg:Alternating} for different values $m = 3,\dots,10$. Three setups for Algorithm~\ref{alg:Alternating} were tested: For the red curve (cross markers) and yellow curve (square markers) no regularization has been applied, whereas in the purple curve (triangle markers) we took $\lambda = \mu = 0.1$. The red curve uses the SVD solution as an initial guess (for $B$). For the yellow and purple curve we used a random initial guess (the same for both). In all cases only $N_{\text{outer}} = 5$ iterations were performed.

As expected, the SVD solution produces a constant error $\| \op - A \otimes B \|_2 = 1$. The other curves follow a predicted algebraic decay. As a curiosity, the purple curve, for which regularization was used, coincides exactly with the error bound $\sigma_1 / (m-1)$. In fact, due to the diagonal structure of $\op$, any $A \otimes B = \beta A_2 \otimes B_2$ with $1 - \beta \le \sigma_1/(m-1)$ achieves this approximation error, and the algorithm indeed returned such solutions with $1 - \beta = \sigma_1/(m-1)$ because with regularization it also aims at $A$ and $B$ with small Frobenius norms. We did not investigate the effect in detail, but it occurs for a certain range of $\lambda$ and~$\mu$. The red and yellow curves in Figure~\ref{fig: experiment1}, obtained without regularization, produce even better approximations and are on top of each other. This shows that in this example the choice of the initial guess does not seem to have a big influence and may even suggest that the obtained solutions could be globally optimal. 

\subsection{Approximation of random operators}

In this experiment we apply the algorithms to a matrix $\op \in \R^{mn \times mn}$ with random Gaussian entries (normalized to $\| \op \|_2 = 1$) and for a sequence of target Kronecker ranks $k=1,2,\dots,\min(m^2,n^2)$. Recall that with $k = \min(m^2,n^2)$ an exact decomposition exists, so the error should be zero.

Figure~\ref{fig: experiment2} shows computed approximation errors $\| \op - \sum_{j=1}^k A_j \otimes B_j \|_2$ for the case $m= 4$, $n = 5$, and $k=1,2,\dots, 16$. The ASDP curves are obtained using the same setup as in section~\ref{sec: experiment1} (again $N_{\text{outer}} = 5$), except that for the purple curve (triangle markers) the regularization parameters were adapted to $\lambda = \mu = 0.1/k$. Decreasing the regularization with $k$ is necessary to account for the effect that the number of penalized terms $\|A_j\|_F$ and $\| B_j \|_F$ grows while the norms $\|\op - \sum_{j=1}^k A_j \otimes B_j \|_2$ remain bounded (and are even intended to decrease). Hence for larger $k$ the regularization terms would dominate in the optimization problem, and indeed in our experiments we observed that using the same $\lambda$ and $\mu$ for all $k$ did not lead to a descending approximation error.

\begin{figure}[t]
\includegraphics[width=.67\textwidth]{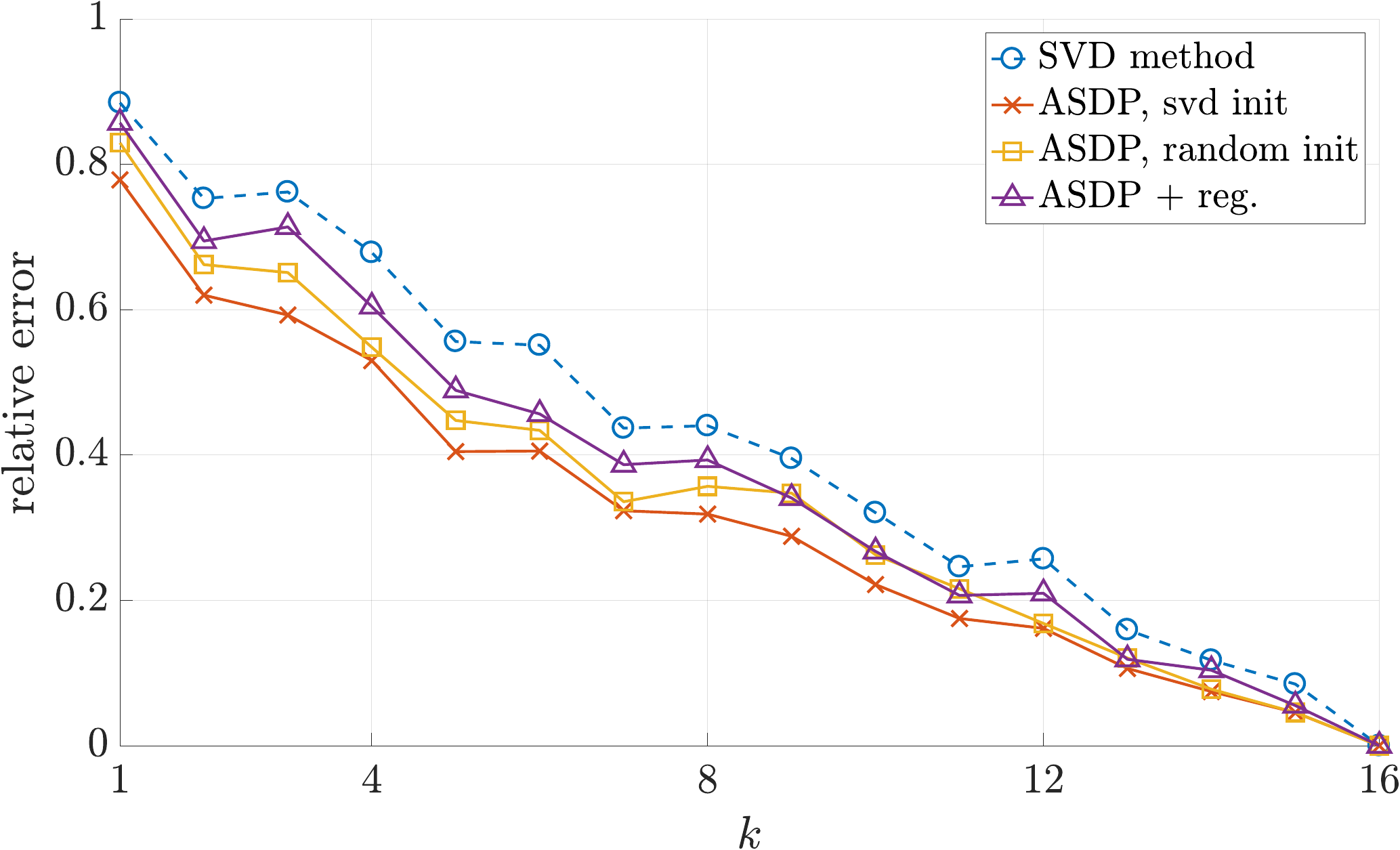}
\caption{Comparison of Algorithm~\ref{alg:svd} (SVD method) and Algorithm~\ref{alg:Alternating} (ASDP) for a random $\op \in \R^{mn \times mn}$ with $m = 4$, $n=5$, scaled to $\| \op \|_2 = 1$. The $y$-axis shows the error $\| \op - \sum_{j=1}^k A_j \otimes B_j \|_2$ of computed approximations for $k=1,2,\dots,16$.}
\label{fig: experiment2}
\end{figure}

We note that the ASDP method is able to significantly improve the relative approximation error compared to the SVD method, even when randomly initialized. This last point is particularly relevant when the computation of an SVD of the $m^2 \times n^2$ matrix $\widehat \op$ needs to be avoided.

\subsection{Inverses of operators with small Kronecker rank}\label{sec: inverse Lyapunov}

Following the considerations in section~\ref{sec: generalizations} we present some experimental results for approximating inverses $\op^{-1}$ of operators $\op$ which themselves have small Kronecker rank. This is based on the modified cost function~\eqref{eq: inverse approximation} to which we also add regularization terms. The ASDP algorithm for this problem is almost identical to Algorithm~\ref{alg:Alternating}, the (formally) only change being that $M(\tau,A,B)$ in the subproblems~\eqref{eq: SDP for A with regularization} and~\eqref{eq: SDP for B with regularization} is replaced with $M_{\text{inv}}(\tau,A,B)$ in~\eqref{eq: M_inv}. We note that in the experiments we did not implement the more efficient representation~\eqref{eq: M_inv efficient} of $M_{\text{inv}}$ but instead treated $\op$ as an unstructured operator in~\eqref{eq: M_inv}.

Two scenarios are considered. The first is a Lyapunov operator
\begin{equation}\label{eq: Lyapunov operator}
\op= L \otimes I + I \otimes L,
\end{equation}
where $L$ is positive definite. This is a special case of more general Sylvester-type operators $\op= L_1 \otimes I + I \otimes L_2$ with $L_1$ and $L_2$ positive definite. Such types of operators play an important role in matrix equations and numerical analysis, and it is well known that their inverses admit highly accurate approximations in spectral norm by operators of low Kronecker rank~\cite{Beylkin2002,Grasedyck04}. Specifically, spectral approximation of $\op^{-1}$ by sums of exponentials leads to operators of the form 
\(
S_k = \sum_{j=1}^k w_j \exp(-t_j L_1) \otimes \exp(-t_j L_2),
\)
where the parameters $w_j$ and $t_j$ can be chosen such that the error to $\op^{-1}$ satisfies
\[
\| \op^{-1} - S_k \|_2 \le C e^{-c k}.
\]
Here the optimal choice of the parameters, as well as the resulting constants $c$ and $C$, depends on the spectral bounds of $L_1$ and $L_2$, and the constants deteriorate with a growing condition number. If only a positive lower bound on the smallest eigenvalues of $L_1$ and $L_2$ is known (and the largest eigenvalues are potentially unbounded), one still has a subexponential but superalgebraic convergence rate $\| \op^{-1} - S_k \|_2 \le C e^{- c \sqrt{k}}$ with (different) constants independent from the upper bound on the spectrum. We refer to~\cite[section~9.8.2]{Hackbusch2019}.

\begin{figure}[t]
\includegraphics[width=.67\textwidth]{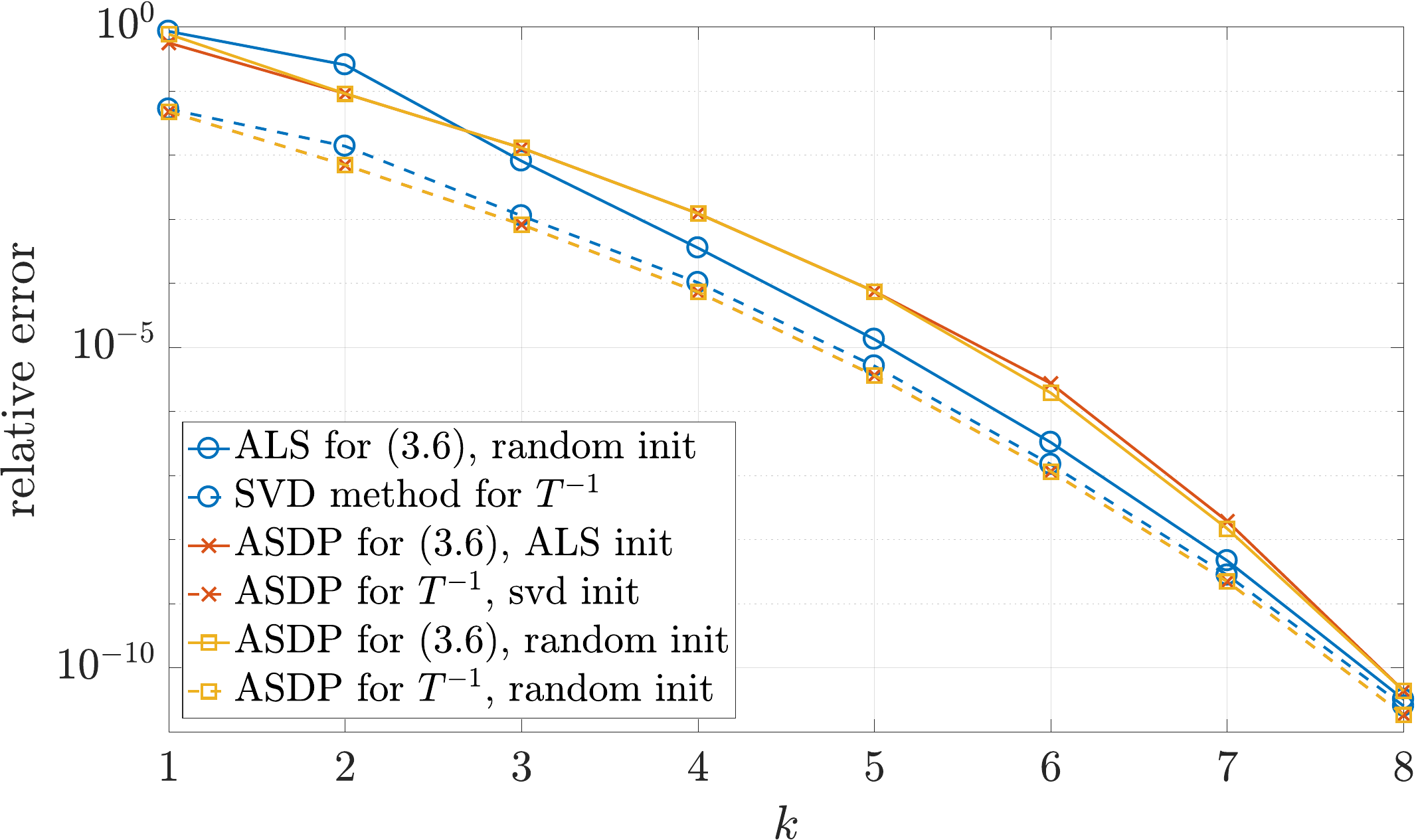}
\caption{Numerical results for approximating the inverse of the Lyapunov operator~\eqref{eq: Lyapunov operator} with $m = n = 10$ and approximation ranks $k=1,2,\ldots,8$. Solid lines correspond to implicit methods based on formulation~\eqref{eq: inverse approximation}. For comparison, dashed lines apply Algorithms~\ref{alg:svd} and~\ref{alg:Alternating} directly to $\op^{-1}$. 
}
\label{fig: experiment3a}
\end{figure}

In the numerical experiment we take $\op$ of the form~\eqref{eq: Lyapunov operator} with $m=n = 10$ and $L = \frac{1}{m-1} \tridiag(-1,2,-1)$ (tridiagonal matrix), which corresponds to a finite-difference discretization of a (negative) second derivative. The results are shown in Figure~\ref{fig: experiment3a}. Depicted are the computed relative errors $\frac{\| \op^{-1} - \sum_{j=1}^k A_j \otimes B_j \|_2}{\| \op^{-1} \|_2}$ for several algorithms and approximation ranks $k=1,2,\dots,8$. The red (cross markers) and yellow (square markers) solid lines are the results of an ASDP method for solving~\eqref{eq: inverse approximation} without regularization ($\lambda = \mu = 0$). The only difference between them is that for the yellow curve a random initialization of the $B_j$ has been taken, whereas for the red curve the $B_j$ were initialized with the solution computed by an alternating least squares (ALS) method for the corresponding problem 
\(
\min_{A_j,B_j} \| \op - \sum_{j=1}^k A_j \otimes B_j \|_F
\) 
in Frobenius norm. This ALS solution is shown as solid blue curve (circle markers). For comparison, we include the results from the SVD and ASDP methods (Algorithms~\ref{alg:svd} and~\ref{alg:Alternating}) when applied directly to the inverse operator $\op^{-1}$, which we explicitly computed for this purpose. In all methods $N_{\text{outer}} = 5$. 

Since the plot is in semilogarithmic scale, the results verify the fact that $\op^{-1}$ is extremely well approximable in spectral norm by sums of Kronecker products. No big differences between the algorithms can be identified. While the implicit approach~\eqref{eq: inverse approximation} clearly works, somewhat surprisingly the ALS method based on the Frobenius norm provides slightly better results. We did not investigate whether this could be caused by internal termination criteria in the SDP solver, but it also does not violate the theory since the approximation of $\op^{-1}$ in~\eqref{eq: inverse approximation} is only an implicit one. We note that in this example $\| \op^{-1} \|_2 \approx 55.5$. Therefore the absolute error $\| \op^{-1} - \sum_{j=1}^k A_j \otimes B_j\|_2$ is larger (at most) by this factor than the depicted curves; see the discussion in section~\ref{sec: generalizations}.

\begin{figure}[t]
\includegraphics[width=.49\textwidth]{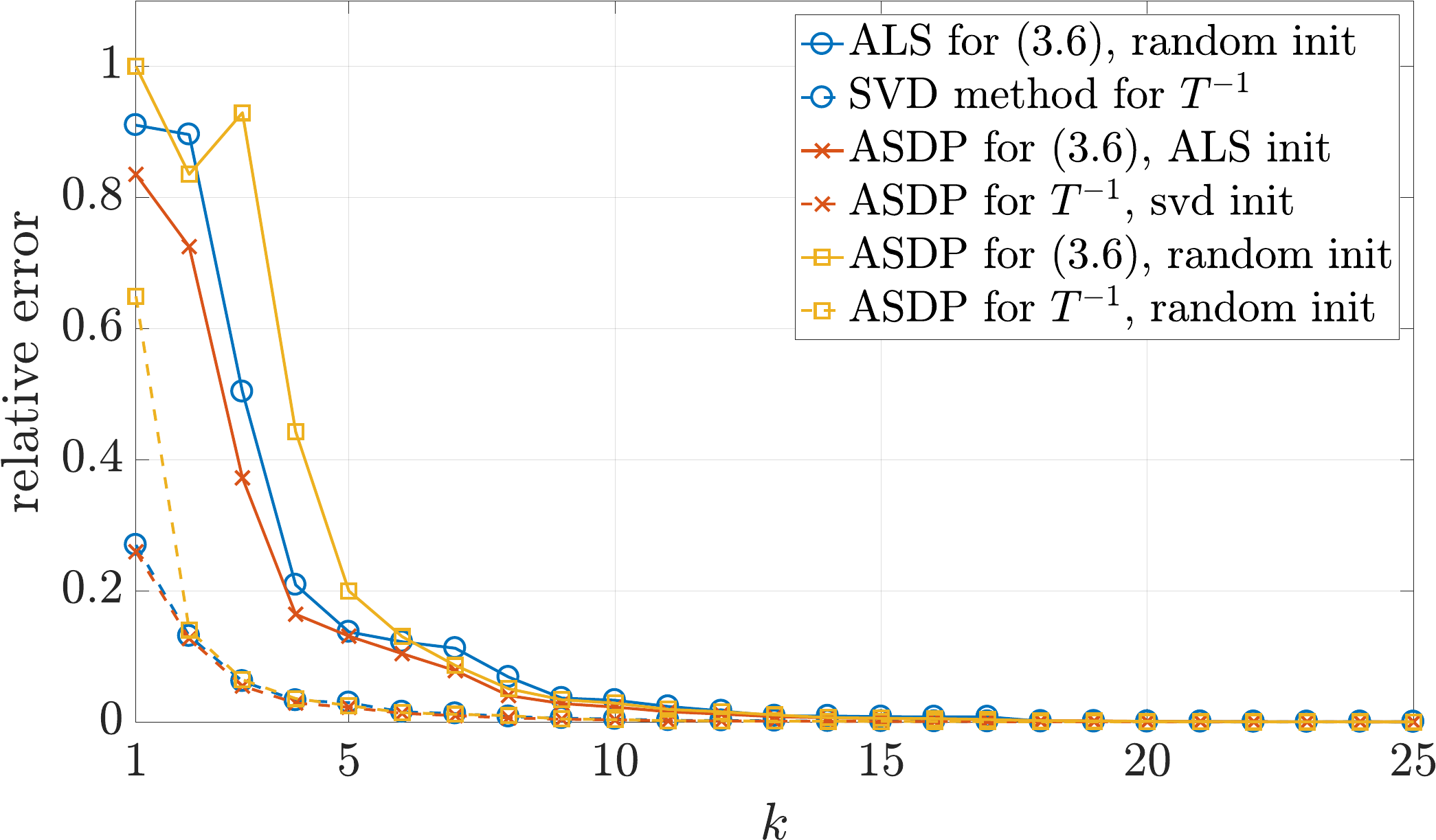}
\includegraphics[width=.48\textwidth]{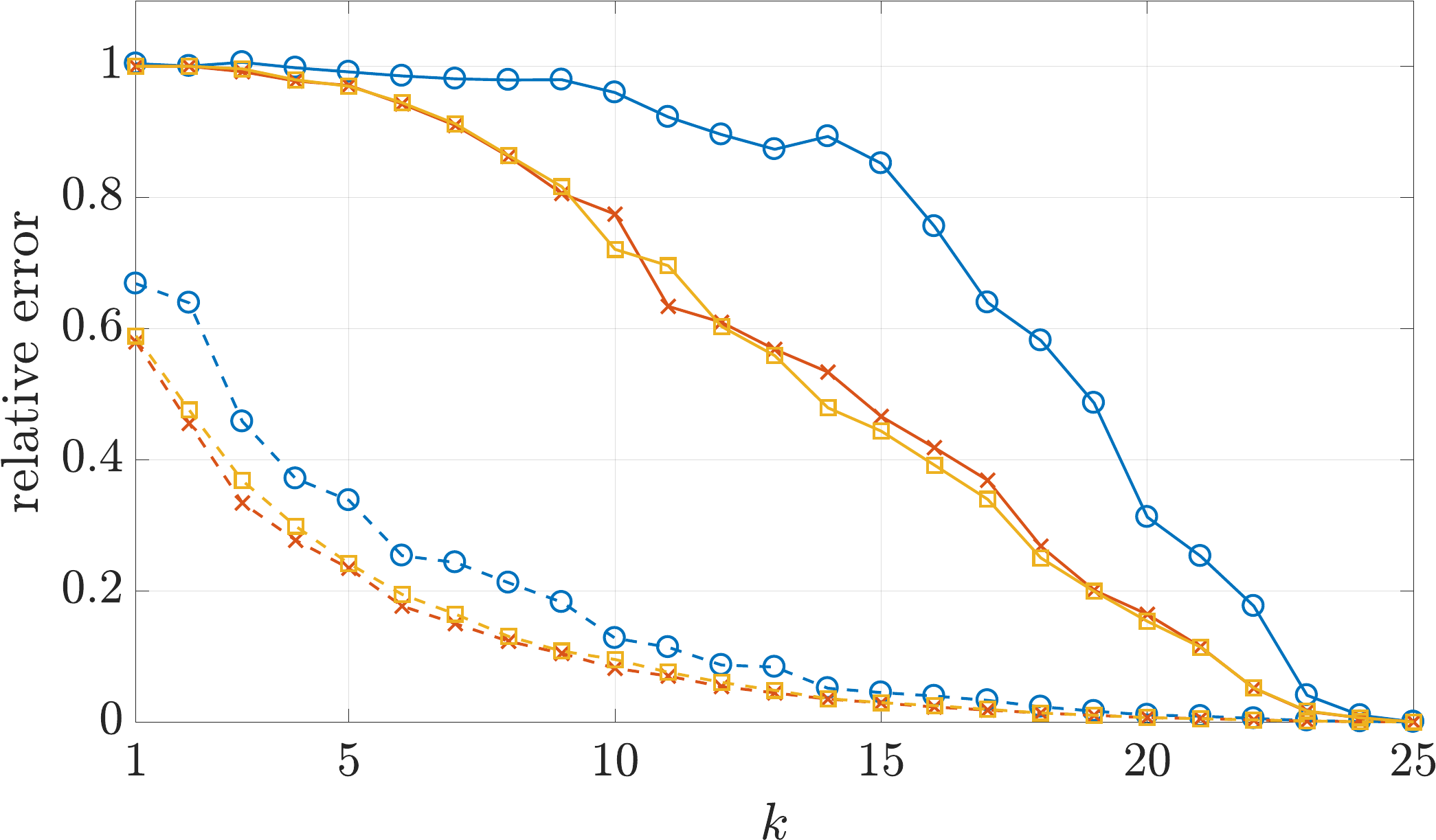}
\caption{Numerical results for approximating the inverse of $\op$ in~\eqref{eq: Kronecker rank 3} with $m = n = 5$ and approximation ranks $k=1,2,\ldots,25$. In the left plot $C_J$, $B_J$ are positive definite, in the right one general. The legend applies to both plots.
}
\label{fig: experiment3bc}
\end{figure}

In the second scenario, we apply the same algorithms to an operator
\begin{equation}\label{eq: Kronecker rank 3}
\op = C_1 \otimes D_1 + C_2 \otimes B_2 + C_3 \otimes D_3
\end{equation}
of Kronecker rank $3$. In this case we are not aware of a rigorous result on the approximability of $\op^{-1}$ by Kronecker products. We take $m=n = 5$ and generate the matrices $C_J$ and $D_J$ randomly; then we scale $\op$ to $\| \op \|_2 = 1$. We distinguish, however, two cases: in the first the $C_J$ and $D_J$ are symmetric positive definite (achieved by replacing them with $C_J^{} C_J^\T$ and $D_J^{} D_J^\T$); in the second they are not. Figure~\ref{fig: experiment3bc} shows the computed relative errors $\frac{\| \op^{-1} - \sum_{j=1}^k A_j \otimes B_j \|_2}{\| \op^{-1} \|_2}$ for $k=1,2,\dots,25$ for both cases. Here we have set $N_{\text{outer}} = 10$ for the alternating optimization methods. The left plot indicates a better error decay in the positive definite case, perhaps even super algebraic. Here $\| \op^{-1} \|_2 \approx 83.4$. In the right plot, the implicit methods (solid lines) do not quite capture the achievable error decay (dashed lines). However, compared to the left plot here the ASDP approach performs significantly better than the ALS method based on Frobenius norm. In this example $\| \op^{-1} \|_2 \approx 203.4$.

\section{Conclusion and outlook}\label{sec: conclusion-outlook}

The problem of approximating a linear operator by sums of Kronecker products in spectral norm is of interest in matrix equations and low-rank calculus. While for the Frobenius norm the approximation problem admits a solution using the SVD, Example~\ref{ex: counterexample} shows that the obtained approximation can be far from optimal in spectral norm and suggests that alternative methods should be studied. In this work, we propose to tackle the approximation problem in spectral norm directly utilizing alternating optimization, where the subproblems can be formulated as SDPs and can therefore be solved in polynomial time to a desired accuracy. The numerical experiments suggest that only a few iterations are necessary to obtain improved approximations using this approach.

This initial work on the subject could be extended in several directions. The presented approach requires the solution of rather high-dimensional SDPs which is computationally expensive and potentially limits the practical applicability. A focus of future work could be to improve the efficiency by fully exploiting structural properties of the operator $\op$ such as sparsity or low (Kronecker) rank in an implementation. In certain situations, it might also be possible to assume a low-rank model for the factor matrices $A_j$ and $B_j$. In such cases one could appeal to the recent literature for (approximately) solving in an efficient manner large SDPs for which one expects low-rank solutions; see, for example,~\cite{Burer2003}. Note that such models, i.e., sums of Kronecker products of low-rank matrices, are of potential interest in quantum entanglement; see, e.g.,~\cite{Gurvits2003,Ni2014,Li2020}.

An extension of the results to complex matrices should be straightforward. We also mention that the alternating SDP approach could be generalized to problems on tensor spaces, that is, for finding low Kronecker rank approximations $\sum_{j=1}^k A_j^{(1)} \otimes \dots \otimes A_j^{(d)}$ to operators on tensor spaces.

\section*{Acknowledgements}
M.D. and V.C. were supported in part by a Max Planck Sabbatical Award.  V.C. was also supported in part by Air Force Office of Scientific Research grants FA9550-20-1-0320 and FA9550-22-1-0225 and by National Science Foundation grant DMS 2113724.
We  thank  the  anonymous  referees  for  their  constructive feedback

\bibliographystyle{amsalpha}
\bibliography{Tensor-Approx}

\providecommand{\bysame}{\leavevmode\hbox to3em{\hrulefill}\thinspace}
\providecommand{\MR}{\relax\ifhmode\unskip\space\fi MR }
\providecommand{\MRhref}[2]{%
  \href{http://www.ams.org/mathscinet-getitem?mr=#1}{#2}
}
\providecommand{\href}[2]{#2}
\begin{thebibliography}{KMP00}

\bibitem[Ber16]{Bertsekas2016}
D.~P. Bertsekas, \emph{Nonlinear programming}, third ed., Athena Scientific,
  Belmont, MA, 2016.

\bibitem[BM02]{Beylkin2002}
G.~Beylkin and M.~J. Mohlenkamp, \emph{Numerical operator calculus in higher
  dimensions}, Proc. Natl. Acad. Sci. USA \textbf{99} (2002), no.~16,
  10246--10251.

\bibitem[BM03]{Burer2003}
S.~Burer and R.~D.~C. Monteiro, \emph{A nonlinear programming algorithm for
  solving semidefinite programs via low-rank factorization}, Math. Program.
  \textbf{95} (2003), no.~2, Ser. B, 329--357.

\bibitem[BV04]{Boyd:Vandenberghe:ConvOpt}
S.~Boyd and L.~Vandenberghe, \emph{Convex optimization}, Cambridge University
  Press, Cambridge, 2004.

\bibitem[Cla75]{Clarke1975}
F.~H. Clarke, \emph{Generalized gradients and applications}, Trans. Amer. Math.
  Soc. \textbf{205} (1975), 247--262.

\bibitem[GPK07]{Gorski2007}
J.~Gorski, F.~Pfeuffer, and K.~Klamroth, \emph{Biconvex sets and optimization
  with biconvex functions: a survey and extensions}, Math. Methods Oper. Res.
  \textbf{66} (2007), no.~3, 373--407.

\bibitem[Gra04]{Grasedyck04}
L.~Grasedyck, \emph{Existence and computation of low {K}ronecker-rank
  approximations for large linear systems of tensor product structure},
  Computing \textbf{72} (2004), no.~3-4, 247--265.

\bibitem[Gur03]{Gurvits2003}
L.~Gurvits, \emph{Classical deterministic complexity of {E}dmond's problem and
  quantum entanglement}, Proceedings of the {T}hirty-{F}ifth {A}nnual {ACM}
  {S}ymposium on {T}heory of {C}omputing, ACM, New York, 2003, pp.~10--19.

\bibitem[Hac19]{Hackbusch2019}
W.~Hackbusch, \emph{Tensor spaces and numerical tensor calculus}, second ed.,
  Springer, Cham, 2019.

\bibitem[HK06a]{Hackbusch2006a}
W.~Hackbusch and B.~N. Khoromskij, \emph{Low-rank {K}ronecker-product
  approximation to multi-dimensional nonlocal operators. {I}. {S}eparable
  approximation of multi-variate functions}, Computing \textbf{76} (2006),
  no.~3-4, 177--202.

\bibitem[HK06b]{Hackbusch2006b}
\bysame, \emph{Low-rank {K}ronecker-product approximation to multi-dimensional
  nonlocal operators. {II}. {HKT} representation of certain operators},
  Computing \textbf{76} (2006), no.~3-4, 203--225.

\bibitem[KMP00]{Konstantinov2000}
M.~Konstantinov, V.~Mehrmann, and P.~Petkov, \emph{On properties of {S}ylvester
  and {L}yapunov operators}, Linear Algebra Appl. \textbf{312} (2000), no.~1-3,
  35--71.

\bibitem[LN20]{Li2020}
Y.~Li and G.~Ni, \emph{Separability discrimination and decomposition of
  {$m$}-partite quantum mixed states}, Phys. Rev. A \textbf{102} (2020), no.~1,
  012402, 8.

\bibitem[L{\"{o}}f04]{Yalmip}
J.~L{\"{o}}fberg, \emph{Yalmip: A toolbox for modeling and optimization in
  {MATLAB}}, In Proceedings of the CACSD Conference (Taipei, Taiwan), 2004,
  Available online at https://yalmip.github.io/.

\bibitem[NQB14]{Ni2014}
G.~Ni, L.~Qi, and M.~Bai, \emph{Geometric measure of entanglement and
  {U}-eigenvalues of tensors}, SIAM J. Matrix Anal. Appl. \textbf{35} (2014),
  no.~1, 73--87.

\bibitem[OST09]{Oseledets2009}
I.~V. Oseledets, D.~V. Savostyanov, and E.~E. Tyrtyshnikov, \emph{Linear
  algebra for tensor problems}, Computing \textbf{85} (2009), no.~3, 169--188.

\bibitem[Sim16]{Simoncini2016}
V.~Simoncini, \emph{Computational methods for linear matrix equations}, SIAM
  Rev. \textbf{58} (2016), no.~3, 377--441.

\bibitem[SS90]{Stewart1990}
G.~W. Stewart and J.~G. Sun, \emph{Matrix perturbation theory}, Academic Press,
  Inc., Boston, MA, 1990.

\bibitem[Stu99]{SeDuMi}
J.~F. Sturm, \emph{Using {S}e{D}u{M}i 1.02, a {MATLAB} toolbox for optimization
  over symmetric cones}, Optim. Methods Softw. \textbf{11/12} (1999), no.~1-4,
  625--653.

\bibitem[Tse01]{Tseng2001}
P.~Tseng, \emph{Convergence of a block coordinate descent method for
  nondifferentiable minimization}, J. Optim. Theory Appl. \textbf{109} (2001),
  no.~3, 475--494.

\bibitem[Tyr04]{Tyrtyshnikov2004}
E.~Tyrtyshnikov, \emph{Kronecker-product approximations for some
  function-related matrices}, Linear Algebra Appl. \textbf{379} (2004),
  423--437.

\bibitem[VL00]{VanLoan2000}
C.~F. Van~Loan, \emph{The ubiquitous {K}ronecker product}, J. Comput. Appl.
  Math. \textbf{123} (2000), no.~1-2, 85--100.

\bibitem[VLP93]{VanLoan1992}
C.~F. Van~Loan and N.~Pitsianis, \emph{Approximation with {K}ronecker
  products}, Linear algebra for large scale and real-time applications
  ({L}euven, 1992), Kluwer Acad. Publ., Dordrecht, 1993, pp.~293--314.

\end{thebibliography}

\bigskip\bigskip

{\small
\noindent
Mareike Dressler, School of Mathematics and Statistics, University of New South Wales, Sydney, NSW 2052, Australia.
\bigskip

\noindent
Andr\'e Uschmajew, Max Planck Institute for Mathematics in the Sciences, 04103 Leipzig, Germany. 

\bigskip

\noindent 
Venkat Chandrasekaran, Department of Computing and Mathematical Sciences and Department of Electrical Engineering, California Institute of Technology, Pasadena, CA 91125, USA.
}

\end{document}